\documentclass{amsart}
\usepackage{graphicx} 
\usepackage{amssymb,amsthm,amsmath,mathtools,float,comment,soul,xcolor,verbatim,graphicx}
\usepackage{hyperref}
\usepackage{cleveref}
\title{Hierarchical Paraproducts}
\author{Oluwadamilola Fasina}
\date{January 2026}

\usepackage[left=3cm, right=3cm, top=2.5cm, bottom=2.5cm]{geometry}
\newtheorem{theorem}{Theorem}[section]
\newtheorem{proposition}[theorem]{Proposition}
\newtheorem{lemma}[theorem]{Lemma}
\newtheorem{corollary}[theorem]{Corollary}
\newtheorem{definition}[theorem]{Definition}
\newtheorem{remark}[theorem]{Remark}

\begin{document}

\maketitle

\begin{abstract}
We outline an extension of paraproduct decompositions for compositions of the form $A(f)$ where $A \in C^{d}(\mathbb{R}), f \in \Lambda_{\alpha}([0,1]^d)$ developed in \cite{fasina2025quasilinearization,fasina2025d} to settings where $(A \in C^1(\mathbb{R}),f \in \Lambda_{\alpha}(X))$  and $ (A \in C^2(\mathbb{R}),f \in \Lambda_{\alpha}(X \times Y))$. To do so, we construct partition trees on $X$ and  $X \times Y$ such that analysis with respect to scale is sensible. We obtain results resembling those of \cite{fasina2025quasilinearization,fasina2025d}, but with the finite sets $X$ and $X \times Y $ as support. In particular we construct the paraproduct $\Pi_{A',A''}^{L,S}: f \to \tilde{A}_{L,S}(f) + \Delta_{L,S}(A,f)$ such that $\Delta_{L,S}(A,f) \in \Lambda_{2\alpha}(X \times Y)$ and $\lVert \Delta_{L,S}(A,f) \rVert_{\Lambda_{2\alpha}(X \times Y)} \leq C_A \lVert f \rVert_{\Lambda_{\alpha}(X \times Y)}$. Analogous results are obtained when the support is just one finite set, $X$. This extension is motivated by situations where one wishes to separate the singular and smooth components of such compositions in graph signal processing environments.
\end{abstract}

\section{Introduction}

Recent endeavors in some areas of graph signal processing \cite{ortega2018graph} have focused on translating tools from harmonic analysis on Euclidean domains \cite{stephane1999wavelet,muscalu2013classical} to graphs and manifolds \cite{saito2022eghwt, cloninger2021natural,gavish2012sampling,gavish2010multiscale}. In particular, \cite{saito2022eghwt,cloninger2021natural,gavish2010multiscale,gavish2012sampling} construct wavelet bases for graphs and \cite{gavish2010multiscale,gavish2012sampling} provide theoretical results concerning the relationship between sparsity and wavelet coefficients in the graph setting. Outlined in this note is an analogous enterprise centered on extending the results of (tensor) paraproducts \cite{bony1981calcul,fasina2025quasilinearization,fasina2025d} for $\alpha$-H\"older functions supported on $[0,1]^d$ to hierarchical paraproducts for $\alpha$-H\"older functions supported on $X$ or $X \times Y$ where $X$ and $X \times Y$ are finite sets with no explicit structure amongst their members. In particular, we obtain the following results on \emph{hierarchical} paraproducts and \emph{hierarchical} tensor paraproducts:

\begin{theorem}
Suppose  $A \in C^1(\mathbb{R})$, $ f \in  \Lambda_{\alpha}(X), 0 < \alpha < \frac{1}{2}$, then for the operator $T: f \to A(f)$ we can approximate $A(f)$ with

\begin{align}
\tilde{A}_L(f) = \sum_{l=0}^L A'(P^l(f))Q^l(f) \in \Lambda_{\alpha}(X)
\label{eq:1}
\end{align}

such that the hierarchical paraproduct transforms $T : f \to A(f)$ to 

\begin{align}
\Pi^{L}_{A'} : f \to \tilde{A}_{L}(f) + \Delta_{L}(A,f)
\label{eq:2}
\end{align}

where $\Delta_L(A,f) \coloneq A(f) - \tilde{A}_L(f) \in \Lambda_{2\alpha}(X)$ is the residual with twice the regularity of $f$ and 

\begin{align}
\lVert \Delta_{L}(A,f) \rVert_{\Lambda_{\alpha}(X)} \leq C_A \lVert f \rVert_{\Lambda_{\alpha}(X)}
\label{eq:3}
\end{align}
\label{thm:1.1}
\end{theorem}

\begin{theorem}
Suppose  $A \in C^2(\mathbb{R})$, $ f \in  \Lambda_{\alpha}(X \times Y), 0 < \alpha < \frac{1}{2}$, then for the operator $T: f \to A(f)$ we can approximate $A(f)$ with

\begin{align}
& \tilde{A}_{L,S}(f) =  \sum_{l=0}^{L} \sum_{s=0}^{S}  A'(P^lP^s(f)) Q^lQ^s(f) + A''(P^lP^s(f))Q^lP^s(f)P^lQ^s(f) \in \Lambda_{\alpha}(X \times Y)
\label{eq:4}
\end{align}

such that the hierarchical tensor paraproduct transforms $T : f \to A(f)$ to 

\begin{align}
\Pi^{L,S}_{A',A''} : f \to \tilde{A}_{L,S}(f) + \Delta_{(L,S)}(A,f)
\label{eq:5}
\end{align}

where $\Delta_{L,S}(A,f) = A(f) - \tilde{A}_{L,S}(f) \in \Lambda_{2\alpha}(X \times Y)$ is a residual with twice the regularity of $f$ and 

\begin{align}
\lVert \Delta_{L,S}(A,f) \rVert_{\Lambda_{2\alpha}(X \times Y)} \leq C_A \lVert f \rVert_{\Lambda_{\alpha}(X \times Y)}
\label{eq:6}
\end{align}
\label{thm:1.2}
\end{theorem}

Theorems \ref{thm:1.1} and \ref{thm:1.2} enumerated above are spiritually similar to the main results of \cite{fasina2025quasilinearization,fasina2025d}, but their explicit constructions are technically distinct. Without the extension included in this note, computing paraproduct decompositions of compositions supported on a finite sets of points not imbued with a relational structure is intractable with the technology developed in \cite{fasina2025quasilinearization,fasina2025d}. Consequently, we (1) show how to construct the hierarchical operators needed to build the approximations $\tilde{A}_L, \tilde{A}_{L,S}$ in Theorems \ref{thm:1.1} and \ref{thm:1.2} (2) include explicit estimates on the regularity of the residual as the regularity arguments are technically distinct from those in \cite{fasina2025quasilinearization,fasina2025d}, though they remain morally compatible. \\

We mention in passing that hierarchical paraproducts for the setting $(f \in \Lambda_{\alpha}(X), A \in C^1(\mathbb{R}))$ can be useful in machine learning situations where $f$ maps from a finite set of data points to labels (e.g. MNIST) and a $C^1$ function is applied to $f$, but one wishes to extract the latent features of $f$ in the presence of this composition c.f. \cite{gavish2010multiscale} and hierarchical tensor paraproducts can be useful in matrix organization settings as detailed by \cite{gavish2012sampling} where, again, knowledge of singularity invariances of the matrix with respect to its organized geometry built on the space of partition trees is obtainable in the presence of $C^2$ functions. Of note, one could consider an extension to this paper for the setting $(f \in \Lambda_{\alpha}(X_1 \times X_2 \times \ldots \times X_d), A \in C^d(\mathbb{R}))$, analogous to the generalization developed in \cite{fasina2025d}.

\section{Acknowledgements}

The author wishes to thank Ronald R. Coifman for his insight.

\section{Preliminaries}

We begin by defining a multiscale partition on a finite set $X = \{ x_1, x_2, \ldots, x_N \}$. The relevant objects are itemized below, following the same notation as \cite{gavish2010multiscale,gavish2012sampling}.

\subsection{Notation and Assumptions}

\begin{enumerate}
    \item $X =  \{x_1, x_2, \ldots, x_N \}$ is a finite set with each $x_i \in \mathbb{R}^m$
    \item $\mathcal{T}_X = \{X^1, X^2, \ldots, X^L\}$ is the set comprised of partitions of $X$ at scales $l=0, \ldots, L$
    \item We say $X_k^l$ is the $k$th node of $X^l$ which is a tree of $X$ generated at scale $l$ such that $X^l \coloneq \cup_{k=1}^{n(l)} X_k^l $ where  $k=1, \ldots , n(l)$ indexes the number of nodes in tree $X^l$
    \item Let $j = 1, \ldots , H(l,k)$ index the set of nodes at the next finest scale whose elements are contained in $X^l_k$ such that $X_k^l \coloneq \cup_{j=1}^{H(l,k)} X_{k,j}^{l+1} $. 
    \item We require that all the nodes $k=1, \ldots, n(l)$ of a partition tree at scale $l$ are mutually disjoint i.e. $X^l_1 \cap X^l_2  \cap \ldots \cap X^l_{n(l)} = \emptyset $
\end{enumerate}

Here we associate the construction of the partition tree, $\mathcal{T}_X$, in steps (1-5) above with the map $\phi(X,l,k,n(l),j,H(l,k)) \to \mathcal{T}_{X}$. We can use $\phi$ to construct an analogous partition tree for the finite set $Y = \{y_1, y_2, \ldots, y_N \}$ for each $y_i \in \mathbb{R}^m $ such that the partition tree for $Y$ is obtained by $\phi(Y,s,r,n(s),i,H(s,r)) \to \mathcal{T}_{Y}$. We refer the reader to \cite{ankenman2014geometry,gavish2012sampling} for a more detailed exposition on explicit constructions of multiscale partition trees on finite sets, but here we use the map $\phi$ to build $\mathcal{T}_Y$ for the sake of readability. The following definitions are used to define $\alpha$-H\"older regularity in this context:

\begin{definition}
Dyadic distance between two points in a set

\begin{align}
\rho_I(x_i, x_j) = \inf_{k,l} 
\begin{cases} 
|X^l_k| &  x_i , x_j \in X^l_k, x_i \neq x_j    \\
0, & x_i = x_j 
\end{cases}
\label{eq:7}
\end{align}

\begin{align}
\rho_R((x_i,y_i) , (x_j,y_j)) = \inf_{k,l,s,r}
\begin{cases} 
|X^l_k \times Y^s_r|, & (x_i,y_i) , (x_j,y_j) \in X^l_k \times Y^s_r,  (x_i,y_i) \neq (x_j,y_j)  \\
0, & (x_i,y_i) = (x_j,y_j) 
\end{cases}
\label{eq:8}
\end{align}

\label{def:3.1}
\end{definition} 

\begin{definition}
We say $f \in \Lambda_{\alpha}(X), \Lambda_{\alpha}(X \times Y) $, respectively, if the following conditions are satisfied, respectively:

\begin{align}
| f(x_i) - f(x_j) | \leq C \rho_I(x_i,x_j)^{\alpha} \hspace{0.2 cm} \forall x_i, x_j \in X
\label{eq:9}
\end{align}

\begin{align}
& |f(x_i,y_i) - f(x_i,y_j) - f(x_j,y_i) + f(x_j,y_j) | \leq C \rho_R((x_i,y_i) , (x_j,y_i))^{\alpha} \rho_R((x_i,y_i) , (x_i,y_j))^{\alpha} \nonumber \\
& |f(x_i,y_i) - f(x_j,y_i) |\leq C  \rho_R((x_i,y_i) , (x_j,y_i))^{\alpha} \nonumber \\
& |f(x_i,y_i) - f(x_i,y_j) |\leq C  \rho_R((x_i,y_i) , (x_i,y_j))^{\alpha}  \nonumber \\
& \forall (x_i,y_i) , (x_j,y_j), (x_i,y_j) , (x_j,y_i)  \in X \times Y
\label{eq:10}
\end{align}

where $C , \alpha > 0$

\label{def:3.2}
\end{definition}

It now becomes profitable to define the Haar-like family (see \cite{gavish2010multiscale,gavish2012sampling} for more details) that we wish to process functions  in $\Lambda_{\alpha}(X)$, $ \Lambda_{\alpha}(X \times Y)$ with.

\begin{definition}
We now define the Haar-like family supported on $X$ and $X \times Y$ following \cite{gavish2010multiscale}. Let the $j$th wavelet $\psi^{l}_{k,j}(x)$ and scaling $\phi^{l}_{k,j}(x)$ functions  associated with the node $X_k^l$ be defined by:

\begin{align}
\psi^{l}_{k,j}(x) = 
\begin{cases} 
1 & X^{l+1}_{k,m}   \\
-1 & X^{l+1}_{k,n}
\end{cases}
\qquad
\phi^{l}_{k,j}(x) = 
\begin{cases} 
1 & X^{l+1}_{k,m}  \\
1 & X^{l+1}_{k,n} 
\end{cases}
\label{eq:11}
\end{align}

where $X^{l+1}_m , X^{l+1}_n \subset X^l_k$ and $m,n \in \{1,2, \ldots, H(l,k)\}$ i.e. $m,n$ are arbitrary indices associated with the nodes in the subtree of $X^l_k$ used to define the $j$th Haar function associated with the node $X^l_k$. Similarly, the $i$th wavelet $\psi^{s}_{r,i}(x) $ and scaling $\phi^{s}_{r,i}(x) $ functions associated with the node $Y^s_r$ for the set $Y$. There are $\tilde{H}(l,k) \coloneq H(l,k) - 1$ wavelet functions associated with the node $X^l_k$ and $\tilde{H}(l,k) \coloneq H(l,k) - 1$ scaling functions associated with node $X^l_k$. Similarly, there are $\tilde{H}(s,r) \coloneq H(s,r) - 1$ wavelet functions associated with the node $Y^s_r$ and $\tilde{H}(s,r) \coloneq H(s,r) - 1$ scaling functions associated with node $Y^s_r$.

\label{def:3.3}
\end{definition}

Definition \ref{def:3.3} permits us to define the orthonormal tensor wavelet, wavelet-scaling, scaling-wavelet and scaling bases supported on all trees $X_k^l \times Y_r^s \subset \mathcal{T}_X \times \mathcal{T}_Y$. 

\begin{definition}

Let the orthonormal hierarchical tensor wavelet, wavelet-scaling, scaling-wavelet and scaling bases associated with $X^l_k \times Y^s_r$ be defined as: $ (\phi^l_{k,j}(x) \times \phi^s_{r,i}(y)), (\psi^l_{k,j}(x) \times \phi^s_{r,i}(y)), (\phi^l_{k,j}(x) \times \psi^s_{r,i}(y)) $ and $(\psi^l_{k,j}(x) \times \phi^s_{r,i}(y))$, respectively such that $ \lVert \phi^l_{k,j}(x) \times \phi^s_{r,i}(y)   \rVert_{L^2(X^l_k \times Y^s_r)} = \lVert \psi^l_{k,j}(x) \times \phi^s_{r,i}(y)   \rVert_{L^2(X^l_k \times Y^s_r)} = \lVert \phi^l_{k,j}(x) \times \psi^s_{r,i}(y)   \rVert_{L^2(X^l_k \times Y^s_r)}= \lVert \psi^l_{k,j}(x) \times \psi^s_{r,i}(y)  \rVert_{L^2(X^l_k \times Y^s_r)} = 1$
\label{def:3.4}
\end{definition}

\begin{remark}
As discussed in \cite{gavish2010multiscale,gavish2012sampling}, such constructions of wavelet families at each scale, $l=0,  \ldots, L$, permit an orthonormal basis to process functions at that scale such that $ \Psi^l = \{ \phi_{k,j}^l, \psi_{k,j}^l \}_{k=1,j=1}^{n(l),\tilde{H}(l,k)}$ is an orthonormal basis for $L^2(X^l)$ functions at scale $l$ and $\Psi_X = \{ \Psi^l \}_{l=0}^L $ is an orthonormal basis for $L^2(X)$. Analogously, one also has $ \Psi^s = \{ \phi_{r,i}^s, \psi_{r,i}^s \}_{r=1,i=1}^{n(s),\tilde{H}(s,r)}$ as the orthonormal basis for $L^2(Y^s)$ such that $\Psi_Y = \{ \Psi_s \}_{s=0}^S $ is an orthonormal basis for $L^2(Y)$. Naturally this leads to $\Psi = \Psi_X \times \Psi_Y$ forming an orthonormal hierarchical tensor haar-like basis for $L^2(X \times Y)$.
\label{rem:3.5}
\end{remark}

To build an approximation for $A(f)$ we need to define the appropriate scaling and wavelet operators, as done in \cite{fasina2025quasilinearization}, except these scaling operators are defined with respect to $\alpha$-H\"older functions supported on partition trees on $X$ and $Y$. As the approximation of $A(f)$ can be constructed with only the scaling, $P^l$ and tensor scaling, $P^lP^s$, operators acting on $f$, we explicitly build those operators with the following definitions:

\begin{definition} 
Let the hierarchical scaling operator $P^l$ acting on $f \in \Lambda_{\alpha}(X)$ be:

\begin{align}
& P^l(f) \coloneq \sum_{k=1}^{n(l)} P^l_k(f) \nonumber \\
& = \sum_{k=1}^{n(l)} \sum_{j=1}^{\tilde{H}(l,k)} P^l_{ k,j}(f) \nonumber \\
& =  \sum_{k=1}^{n(l)} \sum_{j=1}^{\tilde{H}(l,k)} | < f(x), \phi^l_{ k,j}(x) > | \chi_{X^{l+1}_{k,j}} \nonumber \\
& =  \sum_{k=1}^{n(l)} \sum_{j=1}^{\tilde{H}(l,k)} s^l_{k,j} \chi_{X^{l+1}_{k,j}}
\end{align}

where $s^l_{k,j} \coloneq | < f(x), \phi^l_{ k,j}(x)> | =  \frac{1}{|X^{l+1}_{ k,j}|}\sum_{x_m \in X^{l+1}_{k,j}} f(x_m) \phi^l_{ k,j}(x_m)$ is the scaling coefficient for node $X^{l+1}_{k,j}$ and $ \chi_{X^{l+1}_{k,j}} $ is the characteristic function on node $X^{l+1}_{k,j}$.

\label{def:3.6}
\end{definition}

\begin{definition} 
Let the hierarchical tensor scaling operator $P^lP^s$ acting on $f \in \Lambda_{\alpha}(X \times Y)$ be defined as:

\begin{align}
& P^l P^{s}(f) \coloneq \sum_{k=1}^{n(l)} \sum_{r=1}^{n(s)} P^l_k P^s_r(f) \nonumber \\
& = \sum_{k=1}^{n(l)} \sum_{r=1}^{n(s)} \sum_{j=1}^{\tilde{H}(l,m)} \sum_{i=1}^{\tilde{H}(s,r)} P^l_{ k,j} P^s_{r,i}(f) \nonumber \\
& = \sum_{k=1}^{n(l)} \sum_{r=1}^{n(s)} \sum_{j=1}^{\tilde{H}(l,m)} \sum_{i=1}^{\tilde{H}(s,r)}|<  f(x,y), \phi^l_{ k,j}(x) \times \phi^s_{r,i}(y) > |  \chi_{(X^{l+1}_{k,j} \times Y^{s+1}_{r,i})} \nonumber \\
& = \sum_{k=1}^{n(l)} \sum_{r=1}^{n(s)} \sum_{j=1}^{\tilde{H}(l,m)} \sum_{i=1}^{\tilde{H}(s,r)}  \omega^{l,s}_{k,j,r,i}  \chi_{(X^{l+1}_{k,j} \times Y^{s+1}_{r,i})}
\end{align}

where $ \omega^{l,s}_{k,j,r,i} \coloneq |< f(x,y), \phi^l_{ k,j}(x) \times \phi^s_{r,i}(y) > | = \frac{1}{|X^{l+1}_{k,j} \times Y^{s+1}_{r,i}|} \sum_{(x_m,y_p) \in X^{l+1}_{k,j} \times Y^{s+1}_{r,i}} f(x_m,y_p)( \phi^l_{ k,j}(x_m) \times \phi^s_{r,i}(y_p))$ is the tensor scaling coefficient associated with the set $X^{l+1}_{k,j} \times Y^{s+1}_{r,i}$ and $m \in \{1,2, \ldots, |X^{l+1}_{k,j}| \}$ while $p \in \{1,2, \ldots, |Y^{s+1}_{r,i}| \}$

\label{def:3.7}
\end{definition} 

As the expansion coefficients for all hierarchical operators are crucial for defining the approximation and characterizing the residual, we provide explicit definitions of the coefficients and the remaining associated hierarchical operators in the proceeding definitions:

\begin{definition} Define $s^l_{k,j},d^l_{k,j}, \omega^{l,s}_{k,j,s,r}, \beta^{l,s}_{k,j,s,r},\gamma^{l,s}_{k,j,s,r},\alpha^{l,s}_{k,j,s,r}$ to be the expansion coefficients associated with the hierarchical scaling, hierarchical wavelet, hierarchical tensor scaling, hierarchical wavelet-scaling, hierarchical scaling-wavelet, and hierarchical wavelet operators, respectively:

\begin{align}
& s^{l}_{k,j} \coloneq | < f(x) , \phi^l_{ k,j}(x) > | =  \nonumber \\
& \frac{1}{|X^{l+1}_{ k,j}|}\sum_{x_m \in X^{l+1}_{k,j}} f(x_m) \phi^l_{ k,j}(x_m) 
\end{align}

\begin{align}
& d^{l}_{k,j} \coloneq | < f(x), \psi^l_{k,j}(x)> | =  \nonumber \\
& \frac{1}{|X^{l+1}_{ k,j}|}\sum_{x_m \in X^{l+1}_{k,j}} f(x_m) \phi^l_{ k,j}(x_m)
\end{align}

\begin{align}
&  \omega^{l,s}_{k,j,s,r} \coloneq | < f(x,y), \phi^l_{ k,j}(x) \times \phi^s_{r,i}(y) > | =  \nonumber \\
& \frac{1}{|X^{l+1}_{ k,j} \times Y^{s+1}_{ r,i}  |}\sum_{(x_m,y_p)\in X^{l+1}_{k,j} \times Y^{s+1}_{r,i}} f(x_m,y_p)( \phi^l_{k,j}(x_m) \times \phi^s_{r,i}(y_p))
\end{align}

\begin{align}
&  \beta^{l,s}_{k,j,s,r} \coloneq | < f(x,y), \psi^l_{k,j}(x) \times \phi^s_{r,i}(y) > | =  \nonumber \\
& \frac{1}{|X^{l+1}_{ k,j} \times Y^{s+1}_{ r,i}  |}\sum_{(x_m,y_p)\in X^{l+1}_{k,j} \times Y^{s+1}_{r,i}} f(x_m,y_p) (\psi^l_{k,j}(x_m ) \times \phi^s_{r,i}(y_p))
\end{align}

\begin{align}
&  \gamma^{l,s}_{k,j,s,r} \coloneq | < f(x,y), \phi^l_{ k,j}(x) \times \psi^s_{r,i}(y) > | =  \nonumber \\
& \frac{1}{|X^{l+1}_{ k,j} \times Y^{s+1}_{ r,i}  |}\sum_{(x_m,y_p)\in X^{l+1}_{k,j} \times Y^{s+1}_{r,i}} f(x_m,y_p) (\phi^l_{ k,j}(x_m) \times \psi^s_{r,i}(y_p))
\end{align}

\begin{align}
& \alpha^{l,s}_{k,j,s,r} \coloneq | < f(x,y), \psi^l_{k,j}(x) \times \psi^s_{r,i}(y) > | =  \nonumber \\
& \frac{1}{|X^{l+1}_{ k,j} \times Y^{s+1}_{ r,i}  |}\sum_{(x_m,y_p)\in X^{l+1}_{k,j} \times Y^{s+1}_{r,i}} f(x_m,y_p) (\psi^l_{k,j}(x_m ) \times \psi^s_{r,i}(y_p))
\end{align}

where $m \in \{1,2, \ldots, |X^{l+1}_{k,j}| \}$ and $p \in \{1,2, \ldots, |Y^{s+1}_{ r,i }| \}$ index the coordinates of elements in the nodes $X^{l+1}_{k,j}$ and $Y^{s+1}_{ r,i }$, respectively.

\label{def:3.8}
\end{definition}

The remaining associated operators can now be built:

\begin{definition} Let $Q^l, Q^lP^s, P^sQ^l, Q^lQ^s$ be the hierarchical wavelet, wavelet-scaling, scaling-wavelet, and wavelet operators be defined as follows, respectively: 

\begin{align}
Q^l(f) \coloneq \sum_{k=1}^{n(l)} \sum_{j=1}^{\tilde{H}(l,m)}  d^l_{k,j} \chi_{(X^{l+1}_{k,j})}
\end{align}

\begin{align}
Q^lP^s(f) \coloneq \sum_{k=1}^{n(l)} \sum_{r=1}^{n(s)} \sum_{j=1}^{\tilde{H}(l,m)} \sum_{i=1}^{\tilde{H}(s,r)}  \beta^{l,s}_{k,j,r,i}  \chi_{(X^{l+1}_{k,j} \times Y^{s+1}_{r,i})}
\end{align}

\begin{align}
P^lQ^s(f) \coloneq \sum_{k=1}^{n(l)} \sum_{r=1}^{n(s)} \sum_{j=1}^{\tilde{H}(l,m)} \sum_{i=1}^{\tilde{H}(s,r)}  \gamma^{l,s}_{k,j,r,i}  \chi_{(X^{l+1}_{k,j} \times Y^{s+1}_{r,i})}
\end{align}

\begin{align}
Q^lQ^s(f) \coloneq \sum_{k=1}^{n(l)} \sum_{r=1}^{n(s)} \sum_{j=1}^{\tilde{H}(l,m)} \sum_{i=1}^{\tilde{H}(s,r)}  \alpha^{l,s}_{k,j,r,i}  \chi_{(X^{l+1}_{k,j} \times Y^{s+1}_{r,i})}
\end{align}

\label{def:3.9}
\end{definition}

Now we state lemmas concerning $\alpha$-H\"older regularity coefficient decay needed to establish the regularity of the residual. The following lemmas are stated and proved in \cite{gavish2010multiscale,gavish2012sampling}

\begin{lemma}
$g \in \Lambda_{\alpha}(X \times Y) \hspace{0.2 cm} \texttt{iff} \hspace{0.2 cm} |<g,\psi^l_k \times \psi^s_i>| \leq C |X^l_k \times Y^s_r|^{(\alpha + \frac{1}{2})} \hspace{0.2 cm} \forall \psi^l_k \times \psi^s_r \in \Psi^l \times \Psi^s$ and $ \forall X^l_k \times Y^s_r \in \mathcal{T}_X \times \mathcal{T}_Y$
\label{lem:3.10}
\end{lemma}

\begin{proof}
See \cite{gavish2012sampling}
\end{proof}

\begin{lemma}
Suppose $g : X \times Y \to \mathbb{R}$ can be expressed as a sum over expansion coefficients of one of the orthonormal hierarchical  tensor wavelet ($\psi^l_{k,j} \times \psi^s_{r,i}$), wavelet-scaling ($\psi^l_{k,j} \times \phi^s_{r,i}$), scaling-wavelet ($\phi^l_{k,j} \times \psi^s_{r,i}$), and scaling ($\phi^l_{k,j} \times \phi^s_{r,i}$) bases supported on $X \times Y$ as defined in \ref{def:3.3}.

\begin{align}
g(x,y) \coloneq \sum_{l=0}^{L} \sum_{s=0}^{S} \sum_{k=1}^{n(l)} \sum_{r=1}^{n(s)} \sum_{j=1}^{\tilde{H}(l,k)} \sum_{i=1}^{\tilde{H}(s,r)} \eta^{l,s}_{k,j,r,i} \chi_{(X^l_{k,j} \times Y^s_{r,i})}(x,y)
\end{align}

where $\eta^{l,s}_{k,j,r,i}$ is an expansion coefficient of $(\psi^l_{k,j} \times \psi^s_{r,i}), (\psi^l_{k,j} \times \phi^s_{r,i}), (\phi^l_{k,j} \times \psi^s_{r,i})$, or $(\phi^l_{k,j} \times \phi^s_{r,i})$ associated with $X^l_k \times Y^s_r$ and $\chi_{(X^l_{k,j} \times Y^s_{r,i})}$ is the characteristic function on $X^l_{k,j} \times Y^s_{r,i}$. If $|\eta^{l,s}_{k,j,r,i}| \leq |X^l_k \times Y^s_r|^{(2\alpha + 1)} \hspace{0.2 cm} \forall \hspace{0.2 cm} X^l_k \times Y^s_r \subset \mathcal{T}_X \times \mathcal{T}_Y$, then $g \in \Lambda_{2 \alpha}(X \times Y)$

\label{lem:3.11}
\end{lemma}

\begin{proof}
Pick 4 arbitrary points $(x_i,y_i), (x_j,y_i), (x_i,y_j), (x_j,y_j) \in X \times Y$. Let

\begin{align}
(\hat{l},\hat{s}) \coloneq \inf_{\substack{l = 0, \ldots, L \\ s = 0, \ldots, S}} \{ (l,s) : \exists X^l_k \times Y^s_r, \exists (x,y) \in \{ (x_i,y_i), (x_j,y_i), (x_i,y_j), (x_j,y_j)\} ,(x,y) \notin X^l_k \times Y^s_k \}
\end{align}

 i.e. $(\hat{l},\hat{s})$ is the smallest tuple of scales containing at least one cartesian product of nodes in the subtrees which doesn't contain one of the arbitrary points. Let
 
 \begin{align}
 R(\hat{l},\hat{s}) \coloneq \sup_{\substack{k = 1, \ldots, n(l) \\ r = 1, \ldots ,n(s)}} \{ X^{\hat{l}}_k \times Y^{\hat{s}}_r : \exists X^{\hat{l}}_k \times Y^{\hat{s}}_r, \exists (x,y) \in \{ (x_i,y_i), (x_j,y_i), (x_i,y_j), (x_j,y_j)\} ,(x,y) \notin X^{\hat{l}}_k \times Y^{\hat{s}}_k   \} 
 \end{align}

such that $ R(\hat{l},\hat{s})$ is the largest cartesian product of nodes at scales $(\hat{l},\hat{s})$ not containing one of the four points and define

\begin{align}
U(\hat{l},\hat{s}) \coloneq \{ (l,s) : |X^l_k \times Y^s_r| > R(\hat{l},\hat{s}) \hspace{0.2 cm} \forall \hspace{0.2 cm} k=1, \ldots, n(l), r = 1, \ldots, n(s)\}
\end{align}

such that the cartesian product of nodes in the subtrees at the scales $(l,s) \in U(\hat{l},\hat{s})$ contains all 4 points. Define $LS \coloneq \{0,1,2 \ldots, L \} \times \{0,1,2, \ldots, S\}  $ and $B(\hat{l},\hat{s}) \coloneq LS \setminus U(\hat{l},\hat{s}) $ such that $B(\hat{l},\hat{s})$ contains the tuples $(l,s)$ with the associated cartesian products of subtrees where at least one $X^l_k \times Y^s_r$ doesn't contain one of the four points. We show $g \in \Lambda_{\alpha}(X \times Y)$, by satisfying the conditions specified in Definition \ref{def:3.2}. First observe one can write $g$ as:

\begin{align}
&g(x,y) \coloneq \sum_{l=1}^{L} \sum_{s=1}^{S} \sum_{k=1}^{n(l)} \sum_{r=1}^{n(s)} \sum_{j=1}^{\tilde{H}(l,k)} \sum_{i=1}^{\tilde{H}(s,r)} \eta^{l,s}_{k,j,r,i} \chi_{(X^l_{k,j} \times Y^s_{r,i})}(x,y) \nonumber \\
& =  \sum_{(l,s) \in LS} \sum_{k=1}^{n(l)} \sum_{r=1}^{n(s)} \sum_{j=1}^{\tilde{H}(l,k)} \sum_{i=1}^{\tilde{H}(s,r)} \eta^{l,s}_{k,j,r,i} \chi_{(X^l_{k,j} \times Y^s_{r,i})}(x,y) \nonumber \\
& =  \sum_{(l,s) \in U(\hat{l},\hat{s})} \sum_{k=1}^{n(l)} \sum_{r=1}^{n(s)} \sum_{j=1}^{\tilde{H}(l,k)} \sum_{i=1}^{\tilde{H}(s,r)} \eta^{l,s}_{k,j,r,i} \chi_{(X^l_{k,j} \times Y^s_{r,i})}(x,y) +  \sum_{(l,s) \in B(\hat{l},\hat{s})} \sum_{k=1}^{n(l)} \sum_{r=1}^{n(s)} \sum_{j=1}^{\tilde{H}(l,k)} \sum_{i=1}^{\tilde{H}(s,r)} \eta^{l,s}_{k,j,r,i} \chi_{(X^l_{k,j} \times Y^s_{r,i})}(x,y) 
\end{align}

By construction,

\begin{align}
&\sum_{(l,s) \in U(\hat{l},\hat{s})} \sum_{k=1}^{n(l)} \sum_{r=1}^{n(s)} \sum_{j=1}^{\tilde{H}(l,k)} \sum_{i=1}^{\tilde{H}(s,r)} \eta^{l,s}_{k,j,r,i} \chi_{(X^l_{k,j} \times Y^s_{r,i})}(x_i,y_i) - \nonumber \\
&\chi_{(X^l_{k,j} \times Y^s_{r,i})}(x_j,y_i) -  \chi_{(X^l_{k,j} \times Y^s_{r,i})}(x_i,y_j) +  \chi_{(X^l_{k,j} \times Y^s_{r,i})}(x_j,y_j) = 0
\end{align}

and 

\begin{align}
&\sum_{(l,s) \in B(\hat{l},\hat{s})} \sum_{k=1}^{n(l)} \sum_{r=1}^{n(s)} \sum_{j=1}^{\tilde{H}(l,k)} \sum_{i=1}^{\tilde{H}(s,r)} \eta^{l,s}_{k,j,r,i} \chi_{(X^l_{k,j} \times Y^s_{r,i})}(x_i,y_i) - \nonumber \\
&\chi_{(X^l_{k,j} \times Y^s_{r,i})}(x_j,y_i) -  \chi_{(X^l_{k,j} \times Y^s_{r,i})}(x_i,y_j) +  \chi_{(X^l_{k,j} \times Y^s_{r,i})}(x_j,y_j) \nonumber \\
& \leq \sum_{(l,s) \in B(\hat{l},\hat{s})} \sum_{k=1}^{n(l)} \sum_{r=1}^{n(s)} \sum_{j=1}^{\tilde{H}(l,k)} \sum_{i=1}^{\tilde{H}(s,r)} C_{L,S}2^{-(l+s)(2\alpha + 1)} \nonumber \\
& \leq C_{L,S}2^{-(l+s)(2\alpha + 1)} \nonumber \\
& \leq C_{L,S} \rho_R((x_i,y_i) , (x_j,y_i))^{\alpha} \rho_R((x_i,y_i) , (x_i,y_j))^{\alpha}
\end{align}

where the first inequality comes from the estimates on $\eta^{l,s}_{k,j,r,i}$, the second inequality from collapsing the geometric series, and the last inequality from construction of $R(\hat{l},\hat{s})$. Thus, $|g(x_i,y_i) - g(x_i,y_j) - g(x_j,y_i) + g(x_j,y_j)| \leq C_{L,S}\rho_R((x_i,y_i) , (x_j,y_i))^{\alpha} \rho_R((x_i,y_i) , (x_i,y_j))^{\alpha} $. The relations $|g(x_i,y_i) - g(x_j,y_i)| \leq C_{L,S} \rho_R((x_i,y_i) , (x_j,y_i))^{\alpha} $ and  $|g(x_i,y_i) - g(x_i,y_j)| \leq C_{L,S} \rho_R((x_i,y_i) , (x_i,y_j))^{\alpha} $ hold by the same arguments, thus $g \in \Lambda_{\alpha}(X \times Y)$ since selection of the 4 points was arbitrary.

\end{proof}

\section{Hierarchical Paraproducts}

Included in this section are our main results: the hierarchical paraproduct decomposition for $T: f \to A(f)$ under the conditions $A \in C^1(\mathbb{R}),f \in \Lambda_{\alpha}(X)$ where $X = \{x_1, x_2, \ldots, x_N\} $ is a finite set and a second theorem stating the hierarchical tensor paraproduct decomposition under the conditions $A \in C^2(\mathbb{R}),f \in \Lambda_{\alpha}(X \times Y)$ where $X \times Y = \{x_1, x_2, \ldots, x_N\} \times \{y_1, y_2, \ldots, y_N\}$ is the cartesian product of the finite sets $X,Y$.

\begin{theorem}
Suppose  $A \in \mathcal{C}^1(\mathbb{R})$, $ f \in  \Lambda_{\alpha}(X), 0 < \alpha < \frac{1}{2}$, then for the operator $T: f \to A(f)$ we can approximate $A(f)$ with

\begin{align}
\tilde{A}_L(f) = \sum_{l=0}^L A'(P^l(f))Q^l(f) \in \Lambda_{\alpha}(X)
\end{align}

such that the hierarchical paraproduct transforms $T : f \to A(f)$ to 

\begin{align}
\Pi^{L}_{A'} : f \to \tilde{A}_{L}(f) + \Delta_{L}(A,f)
\label{eq:3}
\end{align}

where $\Delta_L(A,f) \coloneq A(f) - \tilde{A}_L(f) \in \Lambda_{2\alpha}(X)$ is the residual with twice the regularity of $f$ and 

\begin{align}
\lVert \Delta_{L}(A,f) \rVert_{\Lambda_{\alpha}(X)} \leq C_A \lVert f \rVert_{\Lambda_{\alpha}(X)}
\end{align}
\label{thm:4.1}
\end{theorem}

\begin{proof}
Proved with the same technology in \cite{fasina2025quasilinearization} which consists of (i) writing $A(f)$ as a telescoping series with respect to the scaling parameter $l$ of the averaging operator, $P^l$, (ii) constructing an interpolation operator permitting one to appeal to the fundamental theorem of calculus (iii) measuring the residual in $L^{\infty}$ to establish $2\alpha$-H\"older regularity. The regularity estimates for the residual, $\Delta_{L}(A,f)$, are acquired with identical techniques deployed for the regularity estimates of the residual, $\Delta_{L,S}(A,f)$, of the paraproduct decomposition stated in Theorem \ref{thm:4.2}, which are handled explicitly in Section 5.
\end{proof}

\begin{theorem}
Suppose  $A \in \mathcal{C}^2(\mathbb{R})$, $ f \in  \Lambda_{\alpha}(X \times Y), 0 < \alpha < \frac{1}{2}$, then for the operator $T: f \to A(f)$ we can approximate $A(f)$ with

\begin{align}
& \tilde{A}_{L,S}(f) =  \sum_{l=0}^{L} \sum_{s=0}^{S}  A'(P^lP^s(f)) Q^lQ^s(f) + A''(P^lP^s(f))Q^lP^s(f)P^lQ^s(f) \in \Lambda_{\alpha}(X \times Y)
\label{eq:2}
\end{align}

such that the hierarchical tensor paraproduct transforms $T : f \to A(f)$ to 

\begin{align}
\Pi^{L,S}_{A',A''} : f \to \tilde{A}_{L,S}(f) + \Delta_{L,S}(A,f)
\label{eq:3}
\end{align}

where $\Delta_{L,S}(A,f) = A(f) - \tilde{A}_{L,S}(f) \in \Lambda_{2\alpha}(X \times Y)$ is a residual with twice the regularity of $f$ and 

\begin{align}
\lVert \Delta_{L,S}(A,f) \rVert_{\Lambda_{2\alpha}(X \times Y)} \leq C_A \lVert f \rVert_{\Lambda_{\alpha}(X \times Y)}
\end{align}
\label{thm:4.2}
\end{theorem}

\begin{proof}

The approximation $\tilde{A}_{L,S}(f)$ is morally alligned with the one developed in \cite{fasina2025quasilinearization}, but computation of the hierarchical tensor operators, $P^lP^s, Q^lP^s, P^lQ^s, Q^lQ^s$, is distinct as the support $f$ is $X \times Y$ and relatedely, measuring $\Delta_{L,S}(A,f)$ in $L^{\infty}$ to establish the $2 \alpha$-H\"older regularity requires more care as multiple Haar-like functions are associated with each node in the partitions of $X \times Y$. Once $P^lP^s, Q^lP^s, P^lQ^s, Q^lQ^s$ are built we follow the steps of \cite{fasina2025quasilinearization} to build the approximation: (1) write $A(f)$ as a telescoping series with respect to the hierarchical tensor scaling operators, $P^lP^s$, (2) recover the approximation by building the interpolation operator, $h_{\mu,\omega}$, defined as:

\begin{align}
& h_{\mu,\omega}(f) \coloneq  h(\omega, \mu, P^{l+1}P^{s+1}(f), P^{l}P^{s+1}(f), P^{l+1}P^{s}(f), P^{l}P^{s}(f)) \nonumber \\
& = \omega( P^{l+1}P^{s}(f) - P^{l}P^{s}(f) ) + \mu  ((P^{l}P^{s+1}(f) +  \nonumber \\
& \omega(P^{l+1}P^{s+1}(f) - P^{l}P^{s+1}(f))) - (P^{l}P^{s}(f) + \omega(P^{l+1}P^{s}(f) - P^{l}P^{s}(f))))
\end{align}

to recover the fundamental theorem of calculus with respect to the scaling parameters $(l,s)$:

\begin{align}
\tilde{A}_{L,S}(f) =  \sum_{l=0}^{L} \sum_{s=0}^{S}  A'(P^lP^s(f)) Q^lQ^s(f) + A''(P^lP^s(f))Q^lP^s(f)P^lQ^s(f)  
\end{align}

(3) compute the residual

\begin{align}
& \Delta_{L,S}(A,f) \coloneq  A(f) - \tilde{A}_{L,S}(f) \nonumber \\
& = \sum_{l=0}^{L} \sum_{s=0}^{S} \int_0^1 \int_0^1 A'(P^lP^s(f) + h_{\mu,\omega}(f))\mathbf{v}_1^{l,s}(f) - A'(P^lP^s(f))\mathbf{v}_1^{l,s}(f) \nonumber \\
& + A''(P^lP^s(f) + h_{\mu,\omega}(f))\mathbf{v}_2^{l,s}(f) -  A''(P^lP^s(f)\tilde{\mathbf{v}}_2^{l,s}(f) d\mu d \omega
\end{align}

where

\begin{align}
\textbf{v}_1^{l,s}(f) \coloneq Q^lQ^{s}(f)
\end{align}

\begin{align}
\textbf{v}_2^{l,s}(f) \coloneq (P^lQ^{s}(f) + \omega Q^lQ^s(f))(Q^lP^{s}(f) + \mu Q^lQ^{s}(f))
\end{align}

\begin{align}
\tilde{\textbf{v}}_2^{l,s}(f) \coloneq Q^lP^{s}(f)P^lQ^{s}(f)
\end{align}

and measure $\Delta_{L,S}(A,f)$ in $L^{\infty}$ to obtain the claim that its regularity is twice as smooth as the data, $f$. Steps (1-2) are identical to \cite{fasina2025quasilinearization} and are therefore omitted, while step (3) is obtained by appealing to Proposition \ref{prop:5.2} which builds on Proposition \ref{prop:5.1} proved in Section 5. Finally the claim, $\lVert \Delta_{L,S}(A,f) \rVert_{2 \alpha}(X \times Y)$ is obtained by observing Corollary \ref{cor:5.3}.

\end{proof}

\section{Residual Estimates}

We show $ \Delta_{L,S}(A,f) \in \Lambda_{2\alpha}(X \times Y)$ by first proving  the following proposition:

\begin{proposition}
\begin{align}
\lVert \Delta_{L,S}(A,f) \lVert_{L^{\infty}( X \times Y )} = \sum_{l=0}^L \sum_{s=0}^S \eta^{l,s}_{k,j,r,i} , |\eta^{l,s}_{k,j,r,i}| \leq 2^{-(l+s)(2\alpha + 1)}
\end{align}
\label{prop:5.1}
\end{proposition}

\begin{proof}

Begin with arbitrary scales $(l,s)$ and begin with the first term in $\Delta(A,f)$ to obtain:

\begin{align}
& \lVert \int_0^1 \int_0^1 A'(P^lP^s(f) + h_{\mu,\omega}(l,s,f))\mathbf{v}_1^{l,s}(f) d\mu d\omega \rVert_{L^{\infty}(X \times Y)} \nonumber \\
& \leq \lVert \sup_{\mu,\omega \in [0,1]} A'(P^lP^s(f) + h_{\mu,\omega}(l,s,f))\mathbf{v}_1^{l,s}(f) \rVert_{L^{\infty}(X \times Y)} \nonumber \\
& = \lVert A'(P^lP^s(f) + Q^lQ^s(f))Q^lQ^s(f) \rVert_{L^{\infty}(X \times Y)}
\end{align}

which comes from computing the supremum of the interpolation operator $h_{\mu,\omega}(f)$ with respect to the scaling parameters $\mu,\omega$ and making the substitution $\mathbf{v}_1^{l,s}(f) \coloneq Q^lQ^s(f)$. Continuing, we have:

\begin{align}
& \lVert A'(P^lP^s(f) + Q^lQ^s(f))Q^lQ^s(f) \rVert_{L^{\infty}(X \times Y)} \nonumber \\
& \leq \sum_{k=0}^{n(l)} \sum_{r=0}^{n(s)} \lVert A'(P_k^lP_r^s(f) + Q_k^lQ_r^s(f))Q_k^lQ_r^s(f) \rVert_{L^{\infty}(X^l_k \times Y^s_r)} \nonumber \\
& = \sum_{k=0}^{n(l)} \sum_{r=0}^{n(s)} \lVert A(P_k^lP_r^s(f)(x,y) + Q_k^lQ_r^s(f))(x,y) - A(P_k^lP_r^s(f)(x',y') + Q_k^lQ_r^s(f)(x',y')) Q_k^lQ_r^s(f) \rVert_{L^{\infty}(X^l_k \times Y^s_r)} \nonumber \\
& \leq \sum_{k=0}^{n(l)} \sum_{r=0}^{n(s)} \lVert ((P_k^lP_r^s(f)(x,y) + Q_k^lQ_r^s(f)(x,y)) - (P_k^lP_r^s(f)(x'
,y') + Q_k^lQ_r^s(f)(x',y'))) Q_k^lQ_r^s(f) \rVert_{L^{\infty}(X^l_k \times Y^s_r)} \nonumber \\
& = \sum_{k=0}^{n(l)} \sum_{r=0}^{n(s)} \sum_{j=1}^{\tilde{H}(l,k)} \sum_{i=1}^{\tilde{H}(s,r)} \lVert ((\omega^{l,s}_{k,j,r,i} + \alpha^{l,s}_{k,j,r,i} ) - (\omega^{l,s}_{k,j,r,i} + \alpha^{l,s}_{k,j,r,i}))\alpha^{l,s}_{k,j,r,i} \chi_{(X^l_{ k,j} \times Y^s_{r,i})}\rVert_{L^{\infty}(X^l_{ k,j} \times Y^s_{r,i})} \nonumber \\
& \leq C_{L,S}  \sum_{k=0}^{n(l)} \sum_{r=0}^{n(s)} \sum_{j=1}^{\tilde{H}(l,k)} \sum_{i=1}^{\tilde{H}(s,r)} 2^{-(l+s)(\alpha + \frac{1}{2})} 2^{-(l+s)(\alpha + \frac{1}{2})}  \nonumber \\
& \leq C_{L,S}2^{-(l+s)(2\alpha + 1)}
\end{align}

where the first step holds from triangle inequality, the second and third steps since $A \in C^2$ and acts on a compact set, and the final step from appealing to estimates on the expansion coefficients of H\"older functions given by Lemma \ref{lem:3.10}. The rest of the terms are estimated using the same arguments dispatched for the first term: Leveraging the regularity of $A$ to bound the action of its derivative by its support and appealing to Lemma \ref{lem:3.10}. We proceed with second term in $\Delta_{L,S}(A,f)$

\begin{align}
& \lVert \int_0^1 \int_0^1 A'(P^lP^s(f))\mathbf{v}_1^{l,s}(f) d\mu d \omega \rVert_{L^{\infty} (X \times Y) } \nonumber \\
& = \lVert A'(P^lP^s(f)) Q^lQ^s(f) d\mu d \omega \rVert_{L^{\infty} (X \times Y) } \nonumber \\
& \leq \sum_{ k=1}^{n(l)} \sum_{r=1}^{n(s)} \lVert A'(P_k^lP_r^s(f)) Q_k^lQ_r^s(f) \rVert_{L^{\infty} (X^l_k \times Y^s_r) } \nonumber \\
& = \sum_{ k=1}^{n(l)} \sum_{r=1}^{n(s)} \lVert (A(P_k^lP_r^s(f))(x,y) - A(P_k^lP_r^s(f))(x',y')) Q_k^lQ_r^s(f) \rVert_{L^{\infty} (X^l_k \times Y^s_r) } \nonumber \\
& \leq \sum_{ k=1}^{n(l)} \sum_{r=1}^{n(s)} \lVert ((P_k^lP_r^s(f))(x,y) - (P_k^lP_r^s(f))(x',y')) Q_k^lQ_r^s(f) \rVert_{L^{\infty} (X^l_k \times Y^s_r) } \nonumber \\
& = \sum_{ k=1}^{n(l)} \sum_{r=1}^{n(s)} \sum_{j=1}^{\tilde{H}(l,k)} \sum_{i=1}^{\tilde{H}(s,r)}  \lVert (\omega^{l,s}_{k,j,r,i} - \omega^{l,s}_{k,j,r,i})\alpha^{l,s}_{k,j,r,i} \chi_{(X^l_{ k,j} \times Y^s_{r,i})}) \rVert_{L^{\infty} (X^l_{ k,j} \times Y^s_{r,i}) } \nonumber \\
& \leq  C_{L,S} 2^{-(l+s)(2\alpha + 1)}
\end{align}

For the third term we get:

\begin{align}
& \lVert \int_0^1 \int_0^1 A''(P^lP^s(f) + h_{\mu,\omega}(l,s,f))\mathbf{v}_2^{l,s}(f) d \mu d \omega \rVert_{L^{\infty}(X \times Y)} \nonumber \\
& \leq \lVert \sup_{\mu, \omega \in [0,1]} \int_0^1 \int_0^1 A''(P^lP^s(f) + h_{\mu,\omega}(l,s,f))\mathbf{v}_2^{l,s}(f) d \mu d \omega \rVert_{L^{\infty}(X \times Y)} \nonumber \\
& = \lVert  A''(P^lP^s(f) + Q^lQ^s(f))(P^lQ^{s}(f) +  Q^lQ^s(f))(Q^lP^{s}(f) +  Q^lQ^{s}(f)) \rVert_{L^{\infty}(X^l_k \times Y^s_r)} \nonumber \\
& \leq \sum_{k=1}^{n(l)} \sum_{r=1}^{n(s)} \lVert  A''(P_k^lP_r^s(f) + Q_k^lQ_r^s(f))(P_k^lQ_r^{s}(f) +  Q_k^lQ_r^s(f))(Q_k^lP_r^{s}(f) +  Q_k^lQ_r^{s}(f)) \rVert_{L^{\infty}(X^l_k \times Y^s_r)} \nonumber \\
& \leq \sum_{k=1}^{n(l)} \sum_{r=1}^{n(s)} \lVert  (P_k^lP_r^s(f)(x,y) + Q_k^lQ_r^s(f)(x,y)) - (P_k^lP_r^s(f)(x',y') + Q_k^lQ_r^s(f)(x',y')) \nonumber \\
& (P_k^lQ_r^{s}(f) +  Q_k^lQ_r^s(f))(Q_k^lP_r^{s}(f) +  Q_k^lQ_r^{s}(f)) \rVert_{L^{\infty}(X^l_k \times Y^s_r)} \nonumber \\
& = \sum_{k=1}^{n(l))} \sum_{r=1}^{n(s)} \sum_{j=1}^{\tilde{H}(l,k)} \sum_{i=1}^{\tilde{H}(s,r)} \lVert  ((\omega^{l,s}_{k,j,r,i} + \alpha^{l,s}_{k,j,r,i}) -(\omega^{l,s}_{k,j,r,i} + \alpha^{l,s}_{k,j,r,i}) )(\gamma^{l,s}_{k,j,r,i} + \alpha^{l,s}_{k,j,r,i})(\beta^{l,s}_{k,j,r,i} + \alpha^{l,s}_{k,j,r,i}) \nonumber \\
&\chi_{(X^l_{ k,j} \times Y^s_{r,i})}) \rVert_{L^{\infty}({\chi_{(X^l_{ k,j} \times Y^s_{r,i})} } ) } \nonumber \\
& \leq  C_{L,S} \sum_{k=1}^{n(l)}\sum_{r=1}^{n(s)} 2^{-(l+s)(3\alpha + 1.5)} \nonumber \\ 
& \leq C_{L,S}2^{-(l+s)(2\alpha + 1)}
\end{align}

and for the fourth term:

\begin{align}
& \lVert A''(P^lP^s(f))\tilde{\mathbf{v}}_2^{l,s}(f) \rVert_{L^{\infty}(X \times Y)} = \lVert A''(P^lP^s(f)) Q^lP^{s}(f)P^lQ^{s}(f) \rVert_{L^{\infty}(X \times Y)} \nonumber \\
& \leq \sum_{k=1}^{n(l)} \sum_{r=1}^{n(s)} \lVert A''(P_k^lP_r^s(f))Q_k^lP_r^{s}(f)P_k^lQ_r^{s}(f)\rVert_{L^{\infty}(X^l_k \times Y^s_r)} \nonumber \\
& = \sum_{k=1}^{n(l)} \sum_{r=1}^{n(s)} \lVert (A'(P_k^lP_r^s(f)(x,y)) -  A'(P_k^lP_r^s(f)(x',y'))) Q_k^lP_r^{s}(f)P_k^lQ_r^{s}(f)\rVert_{L^{\infty}(X^l_k \times Y^s_r)} \nonumber \\
& \leq \sum_{k=1}^{n(l)} \sum_{r=1}^{n(s)} \lVert ((P_k^lP_r^s(f)(x,y)) -  (P_k^lP_r^s(f)(x',y'))) Q_k^lP_r^{s}(f)P_k^lQ_r^{s}(f)\rVert_{L^{\infty}(X^l_k \times Y^s_r)} \nonumber \\
& = \sum_{k=1}^{n(l)} \sum_{r=1}^{n(s)} \sum_{j=1}^{\tilde{H}(l,k)} \sum_{i=1}^{\tilde{H}(s,r)} \lVert (\omega^{l,s}_{k,j,r,i} - \omega^{l,s}_{k,j,r,i})\beta^{l,s}_{k,j,r,i}\gamma^{l,s}_{k,j,r,i} \chi_{(X^l_{ k,j} \times Y^s_{r,i})} \rVert_{L^{\infty}(X^l_{ k,j} \times Y^s_{r,i})} \nonumber \\
&  \leq C_{L,S} \sum_{k=1}^{n(l)} \sum_{r=1}^{n(s)}  \sum_{j=1}^{\tilde{H}(l,k)} \sum_{i=1}^{\tilde{H}(s,r)}  2^{-(l+s)(3 \alpha + 1.5)} \nonumber \\
&  \leq C_{L,S} 2^{-(l+s)(2\alpha + 1)}
\end{align}

Combining the estimates from the preceding equations and summing over scales permits us to write the residual as:

\begin{align}
& \Delta_{L,S}(A,f) \coloneq \sum_{l=0}^{L} \sum_{s=0}^{S} \sum_{k=1}^{n(l)} \sum_{r=1}^{n(s)} \sum_{j=1}^{\tilde{H}(l,k)} \sum_{i=1}^{\tilde{H}(s,r)} \eta^{l,s}_{k,j,r,i} \chi_{(X^l_{k,j} \times Y^s_{r,i})}, | \eta^{l,s}_{k,j,r,i} | \leq 2^{-(l+s)( 2 \alpha+1)} 
\end{align}

where $\eta^{l,s}_{k,j,r,i}$ is the expansion coefficient associated with $X^l_k \times Y^s_r$ and $\chi_{(X^l_{k,j} \times Y^s_{r,i})}$ is the characteristic function on $X^l_k \times Y^s_r$.

\end{proof}

\begin{proposition}
$\Delta_{L,S}(A,f) \in \Lambda_{2\alpha}(X \times Y) $
\label{prop:5.2}
\end{proposition}

This leads us to the following corollary:

\begin{corollary}
$ \lVert \Delta_{L,S}(A,f) \rVert_{\Lambda_{2\alpha}(X \times Y)} \leq C_A \lVert f \rVert_{\Lambda_{\alpha}(X \times Y)} $
\label{cor:5.3}
\end{corollary}

To prove this, we first define the following H\"older norms analogous to those defined in \cite{leeb2016Holder,ankenman2014geometry}

\begin{definition}
\begin{align}
\lVert f \rVert_{\Lambda_{\alpha}(X)} = \sup_{l,k,j} \frac{|<f(x),\psi^l_{k,j}(x)>|}{2^{-l(\alpha + \frac{1}{2})}} 
\qquad
\lVert f \rVert_{\Lambda_{2 \alpha}(X)} = \sup_{l,k,j}  \frac{|<f(x),\psi^l_{k,j}(x)>|}{2^{-l(2\alpha + \frac{1}{2})}} 
\end{align}
\end{definition}

\begin{definition}
\begin{align}
\lVert f \rVert_{\Lambda_{\alpha}(X \times Y)} = \sup_{l,k,j,s,r,i} \frac{|<f(x,y),\psi^l_{k,j}(x) \times \psi^s_{r,i}(y)>|}{2^{-(l+s)(\alpha + \frac{1}{2})}} 
\end{align}

\begin{align}
\lVert f \rVert_{\Lambda_{2 \alpha}(X \times Y)} =  \sup_{l,k,j,s,r,i} \frac{|<f(x,y),\psi^l_{k,j}(x) \times \psi^s_{r,i}(y) >|}{2^{-(l+s)(2\alpha + \frac{1}{2})}} 
\end{align}
\end{definition}

\begin{proof}
Since $f \in \Lambda_{\alpha}(X)$, appealing to Lemma \ref{lem:3.10} gives $\lVert f \rVert_{\Lambda_{\alpha}} \leq C_{L,S} 1$. However,

\begin{align}
& \lVert \Delta_{L,S}(A,f) \rVert_{\Lambda_{2 \alpha}(X \times Y)} = C_{L,S} \frac{2^{-(l+s)(2\alpha +1)}}{2^{-(l+s)(2 \alpha + \frac{1}{2})}} \nonumber \\
& = C_{L,S} 2^{(-(2l\alpha+2s\alpha+l+s) - (-(2l\alpha+2s\alpha+\frac{l+s}{2}))} \nonumber \\
& = C_{L,S} 2^{-\frac{(l+s)}{2}}
\end{align}

where the firs step holds from Proposition \ref{prop:5.1} therefore $\lVert \Delta_{L,S}(A,f) \rVert_{\Lambda_{2 \alpha}(X \times Y)} \leq 2^{-\frac{(l+s)}{2}}$ and $ \Delta_{L,S}(A,f) \rVert_{\Lambda_{2\alpha}(X \times Y)} \leq C_A \lVert f \rVert_{\Lambda_{\alpha}(X \times Y)}  $

\end{proof}

\bibliographystyle{amsplain}
\bibliography{main}

\end{document}